\documentclass[11pt,reqno]{amsart}
\newcommand{\too}{\longrightarrow}
\usepackage[all,2cell]{xy} \UseAllTwocells \SilentMatrices

\newtheorem{thm}{Theorem}[section]

\newtheorem{cor}[thm]{Corollary}
\newtheorem{prp}[thm]{Proposition}

\theoremstyle{definition}
\newtheorem{dfn}[equation]{Definition}

\numberwithin{equation}{section}

\begin{document}

\title[Abelianization of the $F$-divided fundamental group scheme]{Abelianization of the 
$F$-divided fundamental group scheme}

\author[I. Biswas]{Indranil Biswas}

\address{School of Mathematics, Tata Institute of Fundamental
Research, Homi Bhabha Road, Bombay 400005, India}

\email{indranil@math.tifr.res.in}

\author[J. P. dos Santos]{Jo\~ao Pedro P. dos Santos}

\address{Institut de Math\'ematiques de Jussieu -- Paris Rive Gauche, 4 place Jussieu, 
Case 247, 75252 Paris Cedex 5, France}

\email{joao-pedro.dos-santos@imj-prg.fr}



\date{}

\begin{abstract}
Let $(X\, ,x_0)$ be a pointed smooth proper variety defined over an algebraically closed field.
The Albanese morphism for $(X\, ,x_0)$ produces
a homomorphism from the abelianization of the $F$-divided fundamental group
scheme of $X$ to the $F$-divided fundamental group of the Albanese
variety of $X$. We
prove that this homomorphism is surjective with finite kernel. The kernel is also
described.
\end{abstract}

\maketitle

\section{Introduction}

Let $X$ be a proper and smooth variety defined over an algebraically closed field
$k$. Once chosen a $k$-point $x_0$ of $X$, we have the Albanese morphism
$$
\alpha\, :\, X\,\longrightarrow\, A\, ,
$$
as explained in \cite[p.10]{serre1}. Let $\pi_1^\mathrm{et}$ stand for the etale fundamental group as defined on \cite{SGA1} and use the superscript ``ab'' to denote the abelianization.  Since  $\pi_1^{\rm et}(A,\, \alpha(x_0))$
is abelian [SGA1, XI, 2.1, p. 222] the homomorphism of fundamental groups induced by $\alpha$ factors through
the quotient $\pi_1^{\rm et}(X,x_0)\, \longrightarrow\, \pi_1^{\rm et}(X,\,x_0)^\mathrm{ab}$. It is known that
the resulting homomorphism
$$
\alpha_\#\, :\, \pi_1^{\rm et}(X,\,x_0)^\mathrm{ab}\,\longrightarrow\, \pi_1^{\rm et}(A,\, 0)
$$
is surjective with finite kernel; the kernel can also be described (see
\cite[Lemma 5]{katz_lang}). 

Analogous considerations can be made for the essentially finite fundamental group scheme
$\pi_1^{\rm ef}$ \cite{nori1}, since it is known that $\pi_1^{\rm ef}(A,\alpha(x_0))$ is
abelian \cite{nori2}. This work was undertaken by  Antei, who showed that the homomorphism
$\alpha_\#\, :\, \pi_1^{\rm ef}(X,\,x_0)^\mathrm{ab}\,\longrightarrow\, \pi_1^{\rm ef}(A,\, 0)
$ is surjective with finite kernel (see \cite{An}). A. Langer treated the analogous
property in the setting of the $S$-fundamental group scheme \cite{La}.

Our aim here is to address the question for the $F$-divided fundamental group
scheme \cite{ds07} and generalize the analysis made in loc. cit.  More precisely,
we prove the following (see Theorem \ref{13.01.2016--1}):
\begin{itemize}
\item The homomorphism from the abelianized $F$-divided fundamental group scheme of
$X$ to the $F$-divided fundamental group scheme of $A$ (which is
abelian), induced by $\alpha$, is surjective.

\item The kernel of the above surjective homomorphism is finite.
\end{itemize}
We also describe this kernel.

\subsection*{Notation and standard terminology} 

\begin{enumerate}
\item We let $k$ be an algebraically closed field of characteristic $p>0$. 

\item If $Y$ is any $k$-scheme, we let $F\,:\,Y\,\too\, Y$ stand for the \emph{absolute} 
Frobenius morphism: on the underlying topological space it is just the identity and 
$F^\#:\mathcal O_Y\too\mathcal O_Y$ is the usual Frobenius morphism of rings in 
characteristic $p$.

\item The $F$-divided sheaves are the \emph{flat} sheaves of \cite[p. 3,
Definition 1.1]{gieseker}. We adopt the terminology introduced in
\cite[p. 695, Definition 4]{ds07}.

\item The \emph{rank} of an $F$-divided sheaf $\{{\mathcal E}_i\, ,
\sigma_i\}_{i\in\mathbb N}$ is the rank of $\mathcal E_0$.

\item The fundamental group scheme for the $F$-divided sheaves on a pointed smooth 
$k$-scheme $(Y,y_0)$, call it $\Pi(Y,y_0)$, is the one introduced in \cite[p. 696,
Definition 7]{ds07}: the category of $F$-divided sheaves is equivalent, as a Tannakian 
category, to the category of finite dimensional representations of $\Pi(Y,y_0)$.

\item Let $G$ be affine group scheme over $k$. A \emph{quotient} of $G$ is a quotient in 
the sense of Waterhouse \cite[p. 114, 15.1]{waterhouse}. A morphism of affine group schemes $G\to H$ is surjective if $H$ is a quotient of $G$. 

\item Given an affine group scheme $G$ over $k$, we let $G^\mathrm{ab}$, 
$G^\mathrm{uni}$ and $G^{\mathrm{diag}}$ stand respectively for the largest abelian,
largest unipotent, and largest diagonal quotient of $G$.

\item Given any abelian group $\Lambda$, we write $\mathrm{Diag}(\Lambda)$ for the affine group scheme defined in \cite[Part 1, 2.5, p.26]{jantzen}. It is simply the ``torus'' having $\Lambda$ as its group of characters. 

\item If $Y$ is a smooth and proper variety over $k$, its \emph{Picard scheme} will be denoted 
by $\mathbf{Pic}(Y)$. We reserve the symbol $\mathrm{Pic}(Y)$ for the \emph{Picard group}. The connected component of the identity of $\mathbf{Pic}(Y)$ will be denoted by the familiar $\mathbf{Pic}^0(Y)$; we write $\mathrm{Pic}^0(Y)$ for its group of $k$-points. 
At 
present, the best general reference for the Picard scheme seems to be \cite{kleiman}.
\item Finite groups will be identified with their associated finite affine group schemes \cite[2.3]{waterhouse}. The same identifications is extended to profinite groups. 

\item For any abelian group scheme $G$ and any positive integer $m$, we 
write $G[m]$ for the kernel of multiplication by $m$.
 
\item For an abelian variety $B$ over $k$, we let $T_pB$ stand for the 
pro-etale group scheme $\varprojlim B[p^n](k)$. The reader should bear in mind that $T_pB$ is a unipotent group scheme. 
\end{enumerate}

\section{The Albanese variety}

Let $X$ be a proper and smooth variety defined over $k$. In what follows, the underlying reduced subscheme of $\mathbf{Pic}^0(X)$ is denoted by 
$\mathbf{Pic}^0_{\mathrm{red}}(X)$. We fix a closed point $x_0\,\in\, X$. Consider the unique
Poincar\'e line bundle on $X\times\mathbf{Pic}^0(X)$ which is trivial
on $x_0\times \mathbf{Pic}^0(X)$. Restrict the Poincar\'e bundle to
$X\times\mathbf{Pic}^0_{\mathrm{red}}(X)$. Viewing this restriction as a line bundle
on $\mathbf{Pic}^0_{\mathrm{red}}(X)$ parametrized by $X$, we get a morphism
\begin{equation}\label{e1}
\alpha \, :\, X\, \longrightarrow\, 
A\,:=\, \mathbf{Pic}^0(\mathbf{Pic}^0_{\mathrm{red}}(X))\, .
\end{equation}

The \emph{abelian variety} $A$ and the morphism $\alpha$ defined above enjoy a universal 
property characterizing them as the Albanese variety and Albanese morphism in the sense of \cite{Se}.
Since we have no use for this universal property, the reader is only required to bear in 
mind the definition in \eqref{e1}.

\begin{prp}\label{22.07.2015--2}
The homomorphism
\[
\alpha^*\, :\, \mathbf{Pic}^0(A)\,\longrightarrow\, \mathbf{Pic}^0_{\mathrm{red}}(X)
\]
induced by the Albanese morphism $\alpha$ in \eqref{e1} is an isomorphism.
\end{prp}

\begin{proof}
For an abelian variety $B$, we have $\mathbf{Pic}^0(\mathbf{Pic}^0(B))\,=\, B$,
with the identification given by the Poincar\'e line bundle on $B\times\mathbf{Pic}^0(B)$.
More precisely, the classifying morphism $B\,\longrightarrow\,\mathbf{Pic}^0
(\mathbf{Pic}^0(B))$ associated to the above Poincar\'e line bundle is an isomorphism.
In particular, we have $\mathbf{Pic}^0(A)\,=\,\mathbf{Pic}^0_{\mathrm{red}}(X)$.
Now it is straight--forward to check that $\alpha^*$ is the identity morphism
of $\mathbf{Pic}^0_{\mathrm{red}}(X)$.
\end{proof}

\subsection{The group of isomorphism classes of $F$-divided sheaves of rank one}

Define
\begin{equation}\label{22.07.2015--1}
\vartheta(X)\,:=\, \left\{ \begin{array}{c}\text{$F$-divided sheaves } \\
\text{  of rank one on }~X \end{array}\right\}\Bigg/\text{ isomorphisms}
\end{equation}
Under the tensor product of line bundles,  $\vartheta(X)$ is an abelian group. 
Since $X$ is assumed to be proper, the structure of $\vartheta(X)$ is easily
determined from the Picard group of $X$ by means of the following construction \cite[Section 3]{ds07}. 

\begin{dfn}\label{21.07.2015--1} For an abelian group $G$, let $[p]\,:\,G\,\too\,G$
 be the homomorphism $z\,\longmapsto \, p\cdot z$. We define 
$G\langle p\rangle$ as the following projective limit:
\[
\varprojlim\left(\cdots\,\stackrel{[p]}{\longrightarrow}\, G\,
\stackrel{[p]}{\longrightarrow}\, G\,\stackrel{[p]}{\longrightarrow}\, \cdots \right)\, . 
\]
\end{dfn}

For each invertible sheaf $\mathcal L$ of $X$, we write $[\mathcal L]$ for its class in 
the Picard group $\mathrm{Pic}(X)$.

\begin{prp}[{Cf. \cite[p. 7, Theorem 1.8]{gieseker} and \cite[p. 706, Lemma 
20]{ds07}}]\label{21.07.2015--2}
Let $\mathcal L\,=\,({\mathcal L}_0,\, {\mathcal L}_1,\, \cdots)$ be an $F$-divided sheaf of rank one. Write 
\[\tau(\mathcal L)\,=\,
\left([\mathcal L_0],\,[ \mathcal L_1],\,\cdots\right)\,\in
\,\mathrm{Pic}(X)\langle p\rangle\, .
\]
Then, $\tau$ induces an isomorphism between the abelian groups
$\vartheta(X)$ (see \eqref{22.07.2015--1}) and $\mathrm{Pic}(X)\langle p\rangle$ (see
Definition \ref{21.07.2015--1}). 
\end{prp}

\begin{proof}
If $\mathcal L$ and $\mathcal M$ are isomorphic $F$-divided sheaves of rank one, then
clearly $\tau(\mathcal L)\,=\,\tau(\mathcal M)$, and hence $\tau$ is indeed a
homomorphism from $\vartheta(X)$ to $\mathrm{Pic}(X)\langle p\rangle$. It is obvious 
that $\tau$ is surjective, while its injectivity is an observation due to Gieseker 
\cite[p. 6, Proposition 1.7]{gieseker}.
\end{proof}

\begin{cor}\label{22.07.2015--3}
Let $\mathrm{NS}(X)$ stand for the N\'eron--Severi group of $X$ and $\mathrm{NS}(X)'$ for
its subgroup of elements whose order if finite and prime to $p$. Then $\vartheta(X)$ sits in a
short exact sequence 
\[
0\,\longrightarrow\,\mathrm{Pic}^0(X)\langle p\rangle\,\longrightarrow\,\vartheta(X)
\,\longrightarrow\,\mathrm{NS}(X)'\,\longrightarrow\, 0\, .
\]
\end{cor}

\begin{proof}By definition, we have a short exact sequence of abelian groups 
\[
0\,\too\,\mathrm{Pic}^0(X)\,\too\, \mathrm{Pic}(X)\,\too\,\mathrm{NS}(X)\,\too\, 0
\]
which gives rise to an exact sequence of the projective systems considered in
Definition \ref{21.07.2015--1}. Applying the projective limit functor and using that
$$[p]\,:\,\mathrm{Pic}^0(X)\,\too\,\mathrm{Pic}^0(X)$$
is surjective \cite[p.59]{mav}, we conclude that the sequence 
\[
0\,\longrightarrow\,\mathrm{Pic}^0(X)\langle p\rangle\,\longrightarrow\,
\mathrm{Pic}(X)\langle p\rangle\,\longrightarrow\,\mathrm{NS}(X)\langle p\rangle
\,\longrightarrow\, 0
\]
is exact \cite[3.5]{weibel}. As $\mathrm{NS}(X)$ is finitely generated \cite[Exp. 13, Theorem 5.1]{SGA6} and the functor $G\mapsto G\langle p\rangle$ annihilates finitely generated free abelian groups,  the
corollary follows from Proposition \ref{21.07.2015--2} and the easily verified
isomorphisms $\mathrm{NS}(X)'\,\simeq\, \mathrm{NS}(X)'\langle p\rangle\,\simeq\,\mathrm{NS}(X)\langle p\rangle$. 
\end{proof}

\begin{prp}\label{20.07.2015--2}The natural homomorphism
\[
\alpha^*\,:\,\vartheta(A)\,\longrightarrow\, \vartheta(X)
\]
is injective, and its cokernel is the group $\mathrm{NS}(X)'$ introduced in Corollary \ref{22.07.2015--3}. In particular, the kernel of 
\[
\alpha_\#\,:\,\Pi(X,x_0)^{\mathrm{diag}}\,\longrightarrow\,\Pi(A,\alpha(x_0))^\mathrm{diag}
\]
is isomorphic to $\mathrm{Diag}(\mathrm{NS}(X)')$.
\end{prp}

\begin{proof} We have a commutative diagram 
\[
\xymatrix{\vartheta(A)\ar[r]^{\alpha^*}\ar[d]_{\sigma} & \vartheta(X) \ar[d]^{\tau} \\ \mathrm{Pic}(A)\langle p\rangle\ar[r] & \mathrm{Pic}(X)\langle p\rangle \\ \ar[u]^j\mathrm{Pic}^0(A)\langle p\rangle \ar[r] & \ar[u]_i\mathrm{Pic}^0(X)\langle p\rangle.  }
\]
{}From Proposition \ref{21.07.2015--2} we know that $\sigma$ and $\tau$ are isomorphisms. 
{}From Corollary \ref{22.07.2015--3} and \cite[Corollary 2, Ch. IV, Section 19, p. 
165]{mav}, we know that $j$ is an isomorphism. The same Corollary \ref{22.07.2015--3} 
also guarantees that $i$ is injective with co-kernel $\mathrm{NS}(X)'$. Since, by 
Proposition \ref{22.07.2015--2}, the homomorphism $\mathrm{Pic}^0(A)\langle 
p\rangle\,\too\,\mathrm{Pic}^0(X)\langle p\rangle$ is an isomorphism, we are done. The 
statement concerning  group schemes follows easily from the fact that the functor 
$\mathrm{Diag}$ takes exact sequences of abelian groups to exact sequences of affine 
abelian group schemes.
\end{proof}

\section{The group $\pi^{\rm et}(X,x_0)^{\rm ab}$}
Since $\Pi(X,x_0)^{\rm uni}$ is proetale \cite[Corollary 16, p.704]{ds07}, the following considerations are in order. 

\begin{prp}[{Cf.  \cite[p. 308, Lemma 5]{katz_lang}}]\label{20.07.2015--1}
The homomorphism 
\[
\pi_1^{\rm et}(X,\,x_0)^{\mathrm{ab},\mathrm{uni}}\,\longrightarrow \,
\pi_1^{\rm et}(A,\, \alpha(x_0))^{\mathrm{uni}}\,=\,T_p(A)
\]
is surjective and its kernel $K$ is the group of $k$-points of the Cartier dual of the local affine group scheme 
\[
\mathbf{Pic}^{0}(X)/\mathbf{Pic}^0_{\mathrm{red}}(X).
\]
\end{prp}
\begin{proof}This can be easily extracted from \cite[p. 308, Lemma 5]{katz_lang}, but we give
the details for the convenience of the reader. 

We will consider the category of finite groups as a full subcategory of the category of finite group schemes over $k$. In the same spirit, the category of pro-finite groups is regarded as a subcategory of the category of pro-etale group schemes over $k$. 
We know that 
\[
\mathrm{Hom}\left(\pi_1^{\mathrm{et}}(X,x_0)^\mathrm{uni},\, \mathbb Z/p^n\mathbb Z\right)
\,\simeq\, H^1_{\mathrm{et}}(X,\, \mathbb Z/p^n\mathbb Z)\, ,
\]
and the latter is, due to \cite[p. 131, Corollary 4.18]{milne}, isomorphic to 
\[
\mathrm{Hom}\left(\mu_{p^n},\, \mathbf{Pic}(X)\right)\,=\, \mathrm{Hom}\left(\mu_{p^n},
\, \mathbf{Pic}^0(X)\right)\, .
\]
Let $Q$ stand for the quotient $\mathbf{Pic}^{0}(X)/\mathbf{Pic}_{\mathrm{red}}^{0}(X)$. Then, 
using that $\mathrm{Ext}^1(B,\,\mu_m)\,=\,0$ for any abelian variety
$B$ \cite[Remark, p.310]{katz_lang}, we arrive at the
exact sequence 
\[
0\,\too\,\mathrm{Hom}\left(\mu_{p^n},\,\mathbf{Pic}_{\mathrm{red}}^0(X)\right)\,\too\,
\mathrm{Hom}\left(\mu_{p^n},\,\mathbf{Pic}^0(X)\right)\,\too\,
\mathrm{Hom}\left(\mu_{p^n},\, Q\right)\,\too\, 0\, .
\]
Since 
\[
\begin{split}
\mathrm{Hom}\left(\mu_{p^n},\,\mathbf{Pic}^0_{\mathrm{red}}(X)\right)&\,=\,
\mathrm{Hom}\left(\mu_{p^n},\,\mathbf{Pic}^0_{\mathrm{red}}(X)[p^n]\right)\\&\simeq\,
 \mathrm{Hom}\left(\left(\mathbf{Pic}^0_{\mathrm{red}}(X)[p^n]\right)^\vee,\, \mathbb Z/{p^n}\mathbb Z\right)\\
&\simeq \, \mathrm{Hom}\left(A[p^n],\,\mathbb Z/{p^n}\mathbb Z\right)
\end{split}
\]
(see \cite[2.4]{waterhouse} for the first and  \cite[Ch. 15, p. 134, Theorem 1]{mav} for the  second isomorphism), it follows that 
\[\begin{split}
\mathrm{Hom}\left(\mu_{p^n},\,\mathbf{Pic}^0_{\mathrm{red}}(X)\right)&\,\simeq\,
\mathrm{Hom}\left(A[p^n],\,\mathbb Z/{p^n}\mathbb Z\right)\\ &
\,\simeq\, \mathrm{Hom}\left(T_pA,\, \mathbb Z/p^n\mathbb Z\right)
\end{split}
\]
thus completing the proof.
\end{proof}

\section{The main result}
Since the $F$-divided fundamental group $\Pi(A,\alpha(x_0))$ of the Albanese variety $A$ is abelian \cite[p. 707, Theorem 21]{ds07}, we have a commutative diagram 
\[
\xymatrix{\Pi(X,x_0) \ar[r] \ar[d]& \Pi(A,\alpha(x_0))\\ \Pi(X,x_0)^{\rm ab}\ar[ru]_{\alpha_\#}.}
\]

\begin{thm}\label{13.01.2016--1}
Let $\mathrm{NS}(X)'$ be as in Corollary \ref{22.07.2015--3} and $K$ as in Proposition
\ref{20.07.2015--1}. The homomorphism
\[
\alpha_\#\,:\,\Pi(X,\,x_0)^{\rm ab}\,\longrightarrow\, \Pi(A,\,\alpha(x_0))
\]
is surjective, and its kernel is 
\[
\mathrm{Diag}(\mathrm{NS}(X)') \times K.
\]
\end{thm}

\begin{proof}
We can write  $\Pi(X,\,x_0)^\mathrm{ab}$  as a product of $U\times \Delta$,
where $U$ is unipotent and $\Delta$ is diagonal \cite[9.5, p. 70, Theorem]{waterhouse}. By
definition $\Pi(X,\, x_0)^\mathrm{diag}$ is the largest diagonal quotient of
$\Pi(X,\, x_0)^\mathrm{ab}$, and there are no nontrivial homomorphisms from $U$ to $\Delta$ (use \cite[8.3, Corollary, p.65]{waterhouse}). Therefore,
we have $\Pi(X,x_0)^{\mathrm{diag}}\,\simeq\, \Delta$. The same argument shows, employing \cite[Exercise 6, p.67]{waterhouse}, that
$\Pi(X,\,x_0)^{\mathrm{ab},\mathrm{uni}}\,\simeq\, U$. 
As remarked in \cite[p. 707, Theorem 21]{ds07}, we have  isomorphisms
\[
\begin{split}
\Pi(A,\, \alpha(x_0)) &\simeq  \mathrm{Diag}(\mathrm{Pic}^0(A)\langle p\rangle)\times \pi_1^{\rm et}(A,\alpha(x_0))^{\rm uni}\\&\simeq   \mathrm{Diag}(\mathrm{Pic}^0(A)\langle p\rangle)\times T_p(A).
\end{split}
\]
Since there are no nontrivial homomorphisms from a unipotent (respectively, diagonal) affine
group scheme to a diagonal (respectively, unipotent) affine group scheme,
the homomorphism $\alpha_\#$ in the statement  is given by a pair of homomorphisms 
\[
\alpha_\#^{\mathrm{diag}} \times \alpha^{\mathrm{uni}}_\#\,:\,\Pi(X,x_0)^\mathrm{diag}\times
\Pi(X,x_0)^{\mathrm{ab},\mathrm{uni}}\,\longrightarrow \, \mathrm{Diag}(\mathrm{Pic}^0(A)
\langle p\rangle)\times T_p(A)\, .
\]
{}From Corollary \ref{20.07.2015--2}, we know that the
homomorphism $\alpha_\#^{\mathrm{diag}}$ is surjective with kernel
$\mathrm{Diag}\left(\mathrm{NS}(X)'\right)$. From 
\cite[p. 704, Corollary 16]{ds07}, we know that $\Pi(X,\,x_0)^\mathrm{ab,uni}$ is pro-etale, that is, $\Pi(X,\,x_0)^\mathrm{ab,uni}\simeq \pi_1^{\rm et}(X,\,x_0)^{\rm ab, uni}$; applying Proposition \ref{20.07.2015--1} we conclude that 
$\alpha_\#^{\mathrm{uni}}$ is surjective with kernel $K$.
\end{proof}

\section*{Acknowledgements}

The first-named author wishes to thank the Institut de Math\'ematiques de Jussieu -- 
Paris Rive Gauche for hospitality while the work was carried out.


\begin{thebibliography}{ZZZZ}
\bibitem[SGA1]{SGA1} {\it Rev\^etements \'etales et groupe fondamental. }
S\'eminaire de g\'eom\'etrie alg\'ebrique du Bois Marie 1960--61. Directed by A. Grothendieck. With two papers by M. Raynaud. Updated and annotated reprint of the 1971 original. Documents Math\'ematiques 3. Soc. Math.  France, Paris, 2003.

\bibitem[SGA6]{SGA6} {\it Th\'eorie des intersections et th\'eor\`eme de Riemann-Roch,
S\'eminaire de G\'eom\'etrie Alg\'ebrique du Bois-Marie 1966--1967 (SGA 6)},
Dirig\'e par P. Berthelot, A. Grothendieck et L. Illusie. Avec la collaboration de D.
Ferrand, J. P. Jouanolou, O. Jussila, S. Kleiman, M. Raynaud et J. P. Serre,
Lecture Notes in Mathematics, Vol. 225, Springer-Verlag, Berlin-New York, 1971.

\bibitem[An11]{An} M. Antei, On the abelian fundamental group scheme of a family
of varieties, {\it Israel Jour. Math.} {\bf 186} (2011), 427--446.

\bibitem[Gi75]{gieseker}D. Gieseker, Flat vector bundles and the fundamental group in 
non-zero characteristics, {\it Ann. Scuola Norm. Sup. Pisa} {\bf 2} (1975), 1--31.

\bibitem[J87]{jantzen} J. C. Jantzen, \emph{Representations of algebraic groups}. Academic Press, 1987. 


\bibitem[KL81]{katz_lang} N. M. Katz and S. Lang, Finiteness theorems in geometric
classfield theory, {\it Enseign. Math.} {\bf 27} (1981), 285--319.

\bibitem[Kl05]{kleiman} S. Kleiman, The Picard Scheme, in {\it Fundamental algebraic 
geometry}. Mathematical Surveys and Monographs, 123. American Mathematical Society, 
Providence, RI, 2005.

\bibitem[La12]{La} A. Langer, On the S-fundamental group scheme. II, {\it J. Inst.
Math. Jussieu} {\bf 11} (2012), 835--854.

\bibitem[Mu08]{mav}D. Mumford, {\it Abelian Varieties}, Second corrected reprint of
the second edition, Tata Institute of Fundamental Research Studies in Mathematics, 5,
Hindustan Book Agency, New Delhi, 2008.
 
\bibitem[Mi80]{milne} J. S. Milne, {\it \'Etale cohomology}, Princeton Mathematical
Series, 33, Princeton University Press, Princeton, N.J., 1980.

\bibitem[N76]{nori1} M. V. Nori, \emph{On the representations of the fundamental group}, Compositio Math. 33 (1976), no. 1, 29--41.

\bibitem[N83]{nori2} M. V. Nori, \emph{The fundamental group-scheme of an abelian variety}, Math. Ann. 263 (1983), 263--266.

\bibitem[dS07]{ds07} J. P. dos Santos, Fundamental group schemes for stratified sheaves,
{\it Jour. Alg.} {\bf 317} (2007), 691--713.

\bibitem[Se58a]{serre1}J.-P. Serre, Morphismes universels et vari\'et\'e 
d'Albanese, {\it S\'{e}minaire Claude Chevalley}, {\bf 4}, 1958--1959, 
Exp. No. 10, 22 p.

\bibitem[Se58b]{Se} J.-P. Serre, Morphismes universels et 
diff\'{e}rentielles de troisi\`{e}me esp\`{e}ce, {\it S\'{e}minaire 
Claude Chevalley}, {\bf 4}, 1958--1959, Exp. No. 11, 8 p.

\bibitem[Wa79]{waterhouse}W. C. Waterhouse, {\it Introduction to affine group schemes},
Graduate Texts in Mathematics, 66. Springer-Verlag, New York-Berlin, 1979.

\bibitem[We94]{weibel}C. A. Weibel, \emph{An introduction to homological algebra}, Cambridge studies in advances mathematics 38, 1994.

\end{thebibliography}
\end{document}